\documentclass[review]{elsarticle}
\usepackage{fullpage,wrapfig,bm}
\journal{European Journal of Operational Research}
\usepackage{amsmath,amssymb,amsthm,multirow,float,graphicx,bm,multicol}
\usepackage{geometry,tabularx,rotating,algorithm}
\usepackage[noend]{algpseudocode}
\allowdisplaybreaks
\usepackage[]{algorithmicx}
\makeatletter
\def\BState{\State\hskip-\ALG@thistlm}
\makeatother

\newtheorem{theorem}{Theorem}

\newtheorem{ThmDef}{Definition}
\newtheorem{ThmLem}{Lemma}
\floatname{algorithm}{Algorithm}

\newlength\Colsep
\newcolumntype{Y}{>{\centering\arraybackslash}X}
\newcolumntype{Z}{>{\raggedleft\arraybackslash}X}

\begin{document}
	\begin{frontmatter}
		\title{Disaggregated Benders decomposition and lazy constraints for solving the budget-constrained dynamic uncapacitated facility location and network design problem}
		\author{Robin H. Pearce \& Michael Forbes}
		\address{School of Mathematics and Physics, University of Queensland, Australia}
		
		\begin{abstract}
			We present an approach for solving to optimality the budget-constrained Dynamic Uncapacitated Facility Location and Network Design problem (DUFLNDP). This is a problem where a network must be constructed or expanded and facilities placed in the network, subject to a budget, in order to satisfy a number of demands. With the demands satisfied, the objective is to minimise the running cost of the network and the cost of moving demands to facilities. The problem can be disaggregated over two different sets simultaneously, leading to many smaller models which can be solved more easily. Using disaggregated Benders decomposition and lazy constraints, we solve many instances to optimality that have not previously been solved. We use an analytic procedure to generate Benders optimality cuts which are provably Pareto-optimal.
		\end{abstract}
		
	\end{frontmatter}
	
	\section{Introduction}
	
	
	In this paper we apply Benders decomposition to a facility location and network design problem, specifically looking at a number of ways of improving convergence of the algorithm. In particular, we disaggregate the Benders sub-problems, use an alternative to the standard Benders feasibility cuts and analytically construct Benders optimality cuts. We also prove the Pareto-optimality of the analytic Benders cuts and discuss the importance of using Pareto-optimal cuts.
	
	Facility location problems are important in many areas of both industry and government. From deciding the location of stores and warehouses, to important services such as police, fire and health, facility location problems can have a large impact on a population. Equally important are network design problems, such as vehicle routing, or utility network optimisation. While these problems have individually been extensively researched and expanded, the combination of facility location and network design has received less attention.
	
	The facility location problem dates back to the start of the 20th century \citep{Weber1909}, and is the basis of many more detailed problems. Benders decomposition \citep{Benders1962} is an ideal technique for solving facility location problems, particularly the uncapacitated facility location (UFL) problem \citep{Fischetti2016}. Geoffrion and Graves \cite{Geoffrion1974} apply Benders Decomposition to a multicommodity variant of the facility location problem to great effect. Magnanti and Wong \cite{Magnanti1981} explore regular and disaggregated Benders Decomposition, apply it to the UFL problem, and propose an interior point method for accelerating convergence of the algorithm. 
	
	More recently, an efficient implementation of Benders Decomposition for the UFL is demonstrated by Fischetti, Ljubic and Sinnl (2016)\cite{Fischetti2016}. They apply disaggregated Benders decomposition with a number of additional features which are useful, particularly for the UFL. Tang, Jiang and Saharidis \cite{Tang2013} use disaggregated Benders decomposition to solve a capacitated facility location problem where the capacities could be modified for a cost. They also consider adding extra constraints to enforce feasibility and tighten the lower bound on the objective value, which are important in the application of Benders Decomposition.
	
	Network flow and design problems have also been a major area of study over the last century. Today, many efficient methods for finding the maximum flow through a network exist \cite{Ford1956,Elias1956}. As such, more recent studies tend to focus on network design problems, where the network itself is optimised to achieve some goal, such as maximising the throughput of the network over time. Many of these problems are excellent candidates for Benders decomposition. Nurre, Cavdaroglu and Wallace \cite{Nurre2012} consider a problem where a utilities network has been partially destroyed, and the reconstruction must be scheduled to maximise total throughput of the network over time. Boland et. al. \cite{Boland2013} find the optimal maintenance schedule of a network, also to maximise throughput.
	
	Another goal may be to minimise the running cost of a network, which is often coupled with a facility- or hub-location aspect. One example of this is the Hub Line Location problem, considered by de S\'{a} et. al. \cite{deSa2015}, where hub facilities must be built in a public transit network and connected in a line. The objective is to minimise the weighted travel time of all demands through the network. Another example is the Uncapacitated Multiple Allocation Hub Location problem considered by Camargo, Miranda Jr. and Luna \cite{Camargo2008}, where hubs must be built so demands can be routed between locations via hubs. de S\'{a}, de Camargo and de Miranda \cite{deSa2013} apply Benders decomposition to another hub location problem, with a number of improvements such as a ``warm start'', disaggregation of the sub-problems and modified feasibility cuts. Both studies apply Benders decomposition to their problems.
	
	We are considering the budget-constrained Dynamic Uncapacitated Facility Location and Network Design Problem (DUFLNDP) presented by Ghaderi and Jabalameli \cite{Ghaderi2013}. The government sets a fixed budget every year for the construction of new health clinics and roads, and one must work within that budget to minimise the running cost of the network while satisfying all demand for health services by routing demand through the network to health clinics.
	
	The remainder of this paper is structured as follows: Section \ref{SecModel} contains our reformulation of the DUFLNDP, which is the base model to which we apply disaggregated Benders decomposition in Section \ref{SecDBD}. This section also covers many details around the implementation of disaggregated Benders decomposition such as Pareto-optimality of Benders optimality cuts and feasibility of sub-problems. In Section \ref{SecWarm} we describe the use of a warm start with Benders decomposition to improve the initial LP-bound. Our computational results are in Section \ref{SecResults}, followed by a discussion of disaggregation in Section \ref{SecDiscussion}, before concluding with Section \ref{SecConclusion}.

	\section{Model Formulation}\label{SecModel}
	
	Ghaderi and Jabalameli \cite{Ghaderi2013} introduce the budget-constrained Dynamic Uncapacitated Facility Location and Network Design problem (DUFLNDP), which is defined on a network of locations. Every location is a client, and all have the potential to host a facility for servicing clients. There is a set of potential links between locations, on which arcs of the network can be constructed. 
	
	The problem covers a number of time periods. At each time there are budgets for opening new facilities and links. Open facilities and links also have associated maintenance or operating costs, which, together with the demand routing costs, form the total cost which is to be minimised.
	
	The main assumptions in this problem are:
	
	\begin{itemize}
		\item Facilities and links have unlimited capacity
		\item Once opened, facilities and links will remain open until at least the end of the planning horizon
		\item Facilities and links are opened instantaneously between time periods
	\end{itemize}
	
	Our notation is slightly different from Ghaderi and Jabalameli \cite{Ghaderi2013}, in particular the variable names. We also present a simplified version of the budgetary constraints which achieve the same outcome. The time periods we are optimising over start at $1$, and if a network exists already, we denote that as being at time $0$. We now present the model formulation:
	
	\subsection*{Sets}
	\begin{tabular}{rl}
	\hline $N$ & Set of network nodes. These include clients and facilities\\
	$A$ & Set of network arcs, both existing and potential. $A \subseteq N\times N$\\
	$T$ & Set of time periods \\ \hline
	\end{tabular}
	\subsection*{Parameters}
	\begin{tabular}{rl}
		\hline $d_{kt}$ & Demand of client $k \in N$ at time $t \in T$ \\
		$g_{it}$ & Cost of opening facility at node $i \in N$ at time $t \in T$\\
		$c_{ijt}$ & Cost of constructing arc $(i,j) \in A$ at time $t\in T$\\
		$\rho_{ijt}$ & Cost per unit of routing demand on arc $(i,j)\in A$ at time $t\in T$\\
		$f_{it}$ & Operating cost of open facility $i\in N$ at time $t\in T$\\
		$h_{ijt}$ & Operating cost of open arc $(i,j)\in A$ at time $t\in T$\\
		$\bar{B}_t$ & Available budget for opening facilities at time $t\in T$\\
		$\hat{B}_t$ & Available budget for opening arcs at time $t\in T$\\ \hline	
	\end{tabular}
	\subsection*{Variables}
	\begin{tabular}{rl}
		\hline $W_{it}$ & 1 if facility $i\in N$ is open at time $t\in T$, 0 otherwise\\
		$X_{ijt}$ & 1 if arc $(i,j)\in A$ is open at time $t\in T$, 0 otherwise\\
		$Z_{ijkt}$ & Fraction of demand of client $k\in N$ travelling along arc $(i,j)\in A$ at time $t\in T$\\ 
		$U_{it}$ & 1 if facility $i\in N$ is constructed at time $t\in T$, 0 otherwise\\
		$V_{ijt}$ & 1 if arc $(i,j)\in A$ is constructed at time $t\in T$, 0 otherwise\\
		\hline
	\end{tabular}
	\subsection*{Objective}
	\begin{equation}
	\text {Minimise} \sum_{t\in T}\left( \sum_{i\in N} f_{it} W_{it} + \sum_{k\in N}\sum_{(i,j)\in A} \rho_{ijt}d_{kt}Z_{ijkt} + \sum_{\substack{(i,j)\in A\\i<j}}h_{ijt}X_{ijt} \right)\label{OPOBJ}
	\end{equation}
	\subsection*{Constraints}


	\begin{align}
	W_{kt}+\sum_{j\in N} Z_{kjkt} \geq 1 & & \forall k\in N, \forall t\in T\label{OPC1}\\
	\sum_{j\in N}Z_{jikt} \leq \sum_{j\in N}Z_{ijkt} + W_{it} & & \forall i,k\in N,i\neq k, \forall t\in T\label{OPC3}\\
	Z_{jkkt} = 0 & & \forall k \in N, \forall j \in N, \forall t \in T \label{OPC2}\\
	Z_{ijkt} \leq X_{ijt} & & \forall (i,j) \in A, \forall k\in N, \forall t\in T\label{OPC4}\\
	W_{i,t-1} + U_{it} = W_{it} && \forall i\in N, \forall t\in T\label{OPC6}\\
	X_{ij,t-1} + V_{ijt} = X_{ijt} && \forall (i,j)\in A, \forall t\in T\label{OPC7}\\
	\sum_{t'=1}^t \sum_{i\in N} g_{it'}U_{it'}\leq \sum_{t'=1}^t \bar{B}_{t'} && \forall t\in T\label{OPC8}\\
	\sum_{t'=1}^t \sum_{(i,j)\in A} c_{ijt'}V_{ijt'}\leq \sum_{t'=1}^t \hat{B}_{t'} && \forall t\in T\label{OPC9}\\
	X_{ijt} = X_{jit} & & \forall (i,j) \in A, i < j, \forall t \in T \label{OPC10}\\
	W_{it}\in \{0,1\}, U_{it}\in \{0,1\} && \forall i\in N, \forall t\in T\\
	X_{ijt}\in \{0,1\},V_{ijt}\in \{0,1\},Z_{ijkt}\geq 0 && \forall (i,j)\in A, \forall t \in T, \forall k \in N
	\end{align}
	
	
	The objective function (\ref{OPOBJ}) is the sum of three costs: the facility operating costs, the cost of routing demand to other facilities and the arc operating costs. Constraints (\ref{OPC1}) say that if a node $k$ has an open facility, then it services its own demand. If not, all demand must leave the node. Constraints (\ref{OPC3}) are flow-conservation constraints at the nodes. Constraints (\ref{OPC2}) ensure demand can not be returned to the node of origin, thus eliminating cycles. Constraints (\ref{OPC4}) restrict the routing of demand to open arcs only. Constraints (\ref{OPC6}) and (\ref{OPC7}) control the opening of facilities and arcs based on the relevant construction variables, and constraints (\ref{OPC8}) and (\ref{OPC9}) ensure that the budget is not exceeded in any time period. Finally, constraints (\ref{OPC10}) enforce bi-directionality of the arcs.
	
	\section{Disaggregation and Benders Decomposition}\label{SecDBD}
	In this problem, the variables $Z_{ijkt}$ are continuous, where all others are integer (binary). The constraints which contain the continuous variables are (\ref{OPC1}-\ref{OPC4}), and these constraints are separate for each $k\in N$ and $t\in T$. Thus it is possible to disaggregate the sub-problems by time and facility. A discussion of disaggregation level can be found in Section \ref{SubSecDisaggLevel}. The goal of each sub-problem is to find the cheapest way of servicing the demand of that facility at that time. There are two possibilities for this: either the site is a facility and can service its own demand for free, or the demand is routed to the nearest (cheapest) open facility.
	
	\subsection{Benders Master Problem}\label{SubSecMaster}
	We denote the contribution of the sub-problem $(k,t)$ as $\theta_{kt}$. The master problem is:
	\begin{equation}
	\text {Minimise} \sum_{t\in T}\left( \sum_{i\in N} f_{it} W_{it} + \sum_{k\in N}d_{kt}\theta_{kt} + \sum_{\substack{(i,j)\in A\\i<j}}h_{ijt}X_{ijt} \right)\label{MPOBJ}\tag{MP-OBJ}
	\end{equation}
	Subject to:
	\begin{align}
	W_{i,t-1} + U_{it} = W_{it} && \forall i\in N, \forall t\in T\label{MPC6}\tag{M1}\\
	X_{ij,t-1} + V_{ijt} = X_{ijt} && \forall (i,j)\in A, \forall t\in T\label{MPC7}\tag{M2}\\
	\sum_{t'=1}^t \sum_{i\in N} g_{it'}U_{it'}\leq \sum_{t'=1}^t \bar{B}_t && \forall t\in T\label{MPC8}\tag{M3}\\
	\sum_{t'=1}^t \sum_{(i,j)\in A} c_{ijt'}V_{ijt'}\leq \sum_{t'=1}^t \hat{B}_t && \forall t\in T\label{MPC9}\tag{M4}\\
	\theta_{kt}\geq \textup{BendersOptimalityCut}(m,\mathbf{W},\mathbf{X},k,t) && \forall k \in N,\forall t \in T, \forall m\in\{1,...,M\}\label{MPBC}\tag{BOC}\\
	\textup{BendersFeasibilityCut}(p,\mathbf{W},\mathbf{X}) && \forall p\in\{1,...,P\}\tag{BFC}\label{MPFC}\\
	W_{it}\in \{0,1\},U_{i,t}\in \{0,1\},\theta_{it}\geq 0 && \forall i\in N, \forall t\in T\tag{M5}\\
	X_{ijt}\in \{0,1\}, V_{ijt}\in \{0,1\}&& \forall (i,j) \in A, \forall t \in T\tag{M6}
	\end{align}
	
	Constraints (\ref{MPBC}) represent the disaggregated Benders cuts, which are added as necessary after solving the associated sub-problems, which will be covered in Section \ref{SubSecSubProblems}. Similarly, constraints (\ref{MPFC}) represent the added constraints required for feasible sub-problems. M and P are the number of added Benders optimality and feasibility cuts respectively. For a feasible integer solution, $W^*$ and $X^*$, we solve each of the sub-problems and calculate their dual variables. If necessary, we add more Benders optimality or feasibility cuts.
	
	\subsection{Initial feasibility}\label{SubSecFeasibility}
	For the solution to be feasible, it must be possible to service the demand of every client for every time period. With the current master problem and Benders optimality cuts, it is possible for the master problem to be solved to optimality while one or more sub-problems are infeasible. The standard way of overcoming this is to add Benders feasibility cuts, however these are often ineffective \cite{deSa2013}.

	A second option is to modify the master problem to ensure the sub-problems will always be feasible. Since links and facilities are only constructed, never destroyed, if the network is feasible in the first time period, it will be feasible for every time period. To ensure this happens, we modify the model to make the first time period a special case. The objective, parameters and variables remain unchanged, we only modify some constraints and add new ones. The modified and new constraints are:
	
	\subsubsection*{Constraints}
	\begin{align}
	Z_{ijkt} \leq X_{ijt} & & \forall (i,j) \in A, \forall k\in N, \forall t\in T\label{IFPC4}\tag{\ref{OPC4}a}\\
	X_{ij,t-1} + V_{ijt} = X_{ijt} && \forall (i,j)\in A, i < j, \forall t\in T, t > 2 \label{IFPC7}\tag{\ref{OPC7}a}\\
	X_{ijt} = X_{jit} & & \forall (i,j) \in A, i < j, \forall t \in T, t > 1 \label{IFPC10}\tag{\ref{OPC10}a}\\
	X_{ij,1} + X_{ji,1} \leq 1 & & \forall (i,j) \in A, i < j \label{IFPC11}\\
	X_{ij,0} + V_{ij,1} = X_{ij,1} + X_{ji,1} & & \forall (i,j) \in A, i < j \label{IFPC12}\\
	X_{ij,1} + X_{ji,1} + V_{ij,2} = X_{ij,2} & & \forall (i,j) \in A, i < j \label{IFPC13}\\
	\sum_{\substack{j\in N\\(i,j)\in A}} X_{ij,1} + W_{i,1} \geq 1 & & \forall i \in N \label{IFPC14}\\
	\sum_{i\in N} W_{i,1} \geq 1 & &\label{IFPC15}	
	\end{align}
	
	The modification to constraint (\ref{OPC4}) enforces directionality of arcs in the first time period, (\ref{OPC7},\ref{OPC10}) are modified appropriately, and the addition of (\ref{IFPC11}) ensures only one direction is allowed for that time. Constraints (\ref{IFPC12}-\ref{IFPC13}) handle the budget constraints, to ensure that if a direction is built in the first time period, the opposite direction will be built for free in the second time period.
	
	These modifications allow us to add constraints (\ref{IFPC14}-\ref{IFPC15}), which ensure that each location has either a facility at the location or an arc leaving the location, and that at least one facility must exist, respectively. This way, either a node is a facility, or it is connected to a node which is either a facility, or connected to a node... and so on. This only fails if a cycle occurs where multiple nodes are connected to each other and none have facilities, so cycle-breaking may be necessary. This change in formulation is more useful in the instances when there is no pre-existing network, as when there are fixed elements of the network there is less choice in its design.
	
	\subsection{IIS feasibility cuts}\label{SubSecIIS}
	To handle the case where cycles occur, we add cycle-breaking feasibility cuts, where the sum of facilities in the cycle plus the sum of arcs leaving the cycle must be at least one. Using Gurobi, we compute the Irreducible Inconsistent Subsystem (IIS), which is ``a subset of the constraints and variable bounds of the original model. If all constraints in the model except those in the IIS are removed, the model is still infeasible. However, further removing any one member of the IIS produces a feasible result.''\cite{gurobi}
	
	The IIS is then a collection of capacity constraints on nodes and arcs which, when lifted, make the sub-problem feasible. This leaves us with the nodes and arcs which can be expanded or added to resolve the infeasibility. We then add a feasibility cut of the form:
	\begin{equation}
	\sum_{i\in \textup{IIS}} W_{i0} + \sum_{(a,b)\in \textup{IIS}} X_{ab0} \geq 1
	\end{equation}
	
	This ensures that enough facilities and arcs will be opened that the demand from the infeasible source node can be served.
	
	\subsection{Sub-Problems}\label{SubSecSubProblems}
	If we have a feasible solution for the integer variables $W^*$ and $X^*$, we can solve the sub-problems as a collection of linear programs. Since $d_{kt}$ only depends on $k$ and $t$, we can leave it out of the objective of the sub-problem and instead apply it to the objective of the master problem. The contribution of each sub-problem to the master problem is represented by $\theta_{kt}$. For each $k\in N$ and $t\in T$ we have the sub-problem:
	\begin{equation}
	\text {Minimise}  \sum_{(i,j)\in A} \rho_{ijt}Z_{ijkt} \label{SPOBJ}\tag{SP-OBJ}
	\end{equation}
	Subject to:
	\begin{align}
	-\sum_{j\in N} Z_{kjkt} \leq W^*_{kt}-1&&\label{SPC1}\tag{S1}\\
	\sum_{j\in N}Z_{jikt} - \sum_{j\in N}Z_{ijkt} \leq W^*_{it} & & \forall i\in N\setminus\{k\}\label{SPC3}\tag{S2}\\
	Z_{jkkt} = 0 && \forall j \in N \label{SPC2}\tag{S3}\\
	Z_{ijkt} \leq X^*_{ijt} & & \forall (i,j) \in A\label{SPC4}\tag{S4}\\
	Z_{ijkt} \geq 0 && \forall (i,j)\in A\tag{S5}
	\end{align}
	
	Constraint (\ref{SPC1}) has been rearranged to show its similarity to (\ref{SPC3}), which we will take advantage of when formulating the explicit dual. There are two collections of dual variables that we are interested in: $\gamma_{i}$ for each node constraint (\ref{SPC1},\ref{SPC3}) and $\lambda_{ij}$ for each arc constraint (\ref{SPC4}). These variables then lead to the following dual formulation:
	
	\begin{equation}
	\text {Maximise } \gamma_k- \sum_{i\in N} \gamma_i W^*_{it} -\sum_{(i,j)\in A}\lambda_{ij}X^*_{ijt} \label{DualObj}\tag{D-OBJ}
	\end{equation}
	Subject to:
	\begin{align}
	&\rho_{ijt} + \lambda_{ij} + \gamma_j - \gamma_i \geq 0 & \forall (i,j) \in A, j\neq k \label{ConDual}\tag{D1}\\
	& \lambda_{ij} \geq 0, \gamma_i \geq 0& \forall (i,j)\in A, \forall i\in N \tag{D2}
	\end{align}
	
	The reason that constraints (\ref{ConDual}) do not apply when $j=k$ is because in those cases $\gamma_j$ is replaced by the unbounded dual variable associated with constraints (\ref{SPC2}), and since this dual variable does not appear in the objective function, it can be set to positive infinity, thus ensuring feasibility of the dual constraints for those arcs.
	
	Magnanti and Wong \cite{Magnanti1981} describe a natural interpretation of the dual variables for the UFL. In a similar way, the dual variables of this problem have a natural interpretation. Here each $\gamma_i$ represents the saving associated with opening a facility at location $i$, and $\lambda_{ij}$ is the saving from opening a new arc between $i$ and $j$. Constraint (\ref{ConDual}) ensures that the reduced cost of each arc is non-negative. This formulation yields the following Benders cut:
	
	\begin{equation}
	\theta_{kt} \geq \gamma_k-\sum_{i\in N}\gamma_i W_{it}-\sum_{(i,j)\in A}\lambda_{ij}X_{ijt}\tag{BC}
	\end{equation}
	
	One can solve these sub-problems as linear programs and extract the dual variables provided by the solver in order to construct a Benders cut. Alternatively, one can solve the sub-problems and produce the required dual variables analytically.
	
	\subsection{Analytic solution to Benders sub-problems}\label{SubSecAnalytic}
	Each sub-problem is a shortest path problem. For each location, one must find the cheapest way of servicing its demand, either at a facility at the source or by routing the demand to another location with a facility. Each sub-problem is indexed by $k$ and $t$, where $k$ is the source node and $t$ is the time period. 
	
	Magnanti and Wong \cite{Magnanti1981} note that the analytic dual variables for the UFL problem have a natural interpretation; the dual variables for the DUFLNDP also have a natural interpretation. $\gamma_i$ represents the saving associated with opening a facility at location $i$ and $\lambda_{ij}$ the saving from opening an arc from $i$ to $j$. Magnanti and Wong also demonstrate for the UFL problem that there are two different possible analytic Benders cuts which can be used. The DUFLNDP shares this property, as we now show.
	
	\subsubsection{First analytic Benders cut}
	The algorithm for computing the dual variables can be found in Algorithm \ref{AlgBendOptFirst}. Note that it assumes that the facility at the source node $k$ is closed, as otherwise the solution is trivial. We begin by constructing a shortest path tree from the source location, giving each node a distance $D_i$ from the source node. If there is no path between $k$ and $i$, then $D_i = \infty$. Now, the value of $\gamma_k$ is assigned the length of the shortest path to the nearest open facility $i^*$, that is, $\gamma_k = \textup{min}\{D_i|W^*_{it} = 1\}\equiv D_{i^*}$. For all other nodes, $\gamma_i = \textup{max}(0,\gamma_k-D_i)$. 

	\begin{algorithm}[t]
		\caption{Algorithm for computing dual variables for first analytic Benders optimality cut for sub-problem $(k,t)$, assuming the sub-problem is feasible.}
		\label{AlgBendOptFirst}
		\begin{algorithmic}
			\State Begin with master problem solution $W^*_{it},X^*_{abt},\theta^*_{kt}$ $\forall i\in N, \forall (a,b)\in A$
			\State Compute shortest distance $D_i$ from $k$ to $i$ for all nodes $i\in N\setminus\{k\}$
			\State $\gamma_k = \min\limits_j\{D_j|W^*_{jt}=1,j\in N\setminus \{k\}\}$
			\For {$i\in N$}
			\State $\gamma_i \leftarrow \textup{max}(0,\gamma_k-D_i)$
			\EndFor
			\For {$(i,j)\in A$}
			\If {$X^*_{ijt}=0$}
			\State $\lambda_{ij} \leftarrow \textup{max}(0,\gamma_i-\gamma_j-\rho_{ijt})$
			\Else 
			\State $\lambda_{ij} \leftarrow 0$
			\EndIf
			\EndFor
			\State Add Constraint $\theta_{kt}\geq\bar{\theta}_{kt}-\sum\limits_{i\in N}\gamma_iW_{it}-\sum\limits_{(i,j)\in A}\lambda_{ijt}X_{ijt}$
		\end{algorithmic}
	\end{algorithm}

	\begin{algorithm}[t]
		\caption{Algorithm for computing dual variables for second analytic Benders optimality cut for sub-problem $(k,t)$, assuming the sub-problem is feasible.}
		\label{AlgBendOptSecond}
		\begin{algorithmic}
			\State Begin with master problem solution $W^*_{it},X^*_{abt},\theta^*_{kt}$ $\forall i\in N, \forall (a,b)\in A$
			\State Compute shortest distance $\bar{D}_i$ from $i$ to an open facility for all nodes $i\in N$
			\State Open all arcs $(i,j)\in A$
			\State Compute shortest distance $D_i$ from $k$ to $i$ for all nodes $i\in N\setminus\{k\}$
			\State $\gamma_k = \bar{D}_k$
			\For {$i\in N$}
			\State $\gamma_i \leftarrow \max(0,\min(\gamma_k-D_i,\bar{D}_i))$
			\EndFor
			\For {$(i,j)\in A$}
			\If {$X^*_{ijt}=0$}
			\State $\lambda_{ij} \leftarrow \textup{max}(0,\gamma_i-\gamma_j-\rho_{ijt})$
			\Else 
			\State $\lambda_{ij} \leftarrow 0$
			\EndIf
			\EndFor
			\State Add Constraint $\theta_{kt}\geq\bar{\theta}_{kt}-\sum\limits_{i\in N}\gamma_iW_{it}-\sum\limits_{(i,j)\in A}\lambda_{ijt}X_{ijt}$
		\end{algorithmic}
	\end{algorithm}
	
	Next we calculate the values for the dual variables $\lambda_{ij}$, associated with the arcs $(i,j)\in A$. For all open arcs ($X^*_{ijt}=1$), $\lambda_{ij}=0$. For any closed arc $(i,j)$, $\lambda_{ij} = \textup{max} (0, \gamma_i- \gamma_j- \rho_{ij})$. 
	
	\begin{theorem}
		The dual variables calculated using Algorithm \ref{AlgBendOptFirst} are dual optimal.
	\end{theorem}

	\begin{proof}
	For these dual variables to form a dual feasible solution, they must satisfy the Constraints (\ref{ConDual}). The constraints are trivially satisfied for any arc where $\gamma_i=0$. For all closed arcs, $\lambda_{ij} \geq \gamma_i-\gamma_j-\rho_{ijt}$, which satisfies Constraint (\ref{ConDual}).
	
	For any open arc $(i,j)$, $\lambda_{ij}=0$, so we must show that $\rho_{ij} + \gamma_j - \gamma_i \geq 0$. By the property of the shortest path distances, $D_j \leq D_i + \rho_{ij}$, or $\rho_{ij}+D_i-D_j\geq 0$. We also have, by construction of the dual variables, that $\gamma_j \geq \gamma_k-D_j$, or $D_j \geq \gamma_k - \gamma_j$. Finally, we are only considering where $\gamma_i>0$, and in this case we have that $\gamma_i = \gamma_k - D_i$. Combining these, we get:
	\begin{align*}
	\rho_{ij} + \gamma_j - \gamma_i &= \rho_{ij} + \gamma_j - (\gamma_k - D_i)\\
	&\geq \rho_{ij} + (\gamma_k-D_j) - (\gamma_k-D_i)\\
	&\geq \rho_{ij} + D_i - D_j\\
	&\geq 0
	\end{align*}
	
	So the dual variables obtained from Algorithm \ref{AlgBendOptFirst} are dual feasible. The objective value given by these dual variables is the same as the optimal objective value of the primal problem, which is the length of the shortest path ($D_{i^*}=\gamma_k$), and for all nodes, either $\gamma_i = 0$ or $W^*_{it}=0$, and likewise for arcs. Thus the dual variables form a dual optimal solution, and may be used to add a Benders optimality cut to the master problem.
	\end{proof}

	The cuts generated using these dual variables are Pareto-optimal, which is important for improving convergence of the master problem \cite{Magnanti1981}. We prove they are Pareto-optimal in Section \ref{SubSecPOProof}. In terms of the natural interpretation, we place much of the savings on the arcs, since if facilities beyond closed arcs are opened, no saving will be obtained until the arcs are open. As such, the first analytic Benders cut is referred to as a $\lambda$-heavy cut. While this Benders cut is Pareto-optimal, it may not be the only Pareto-optimal analytic cut which can be derived from a given solution. 
	
	\begin{figure}
		\centering
		\includegraphics[width=0.9\textwidth]{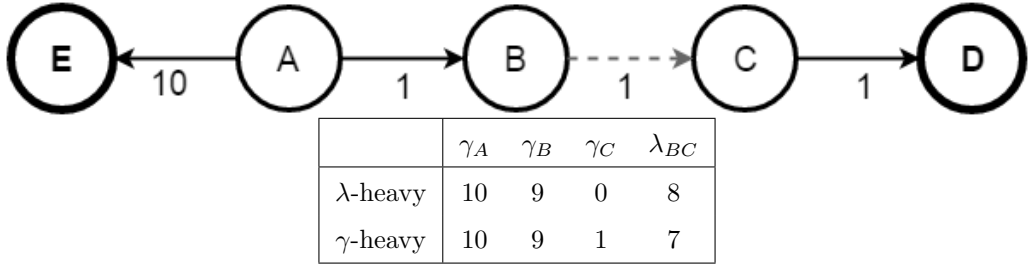}
		\begin{tabular}{|c|cccc|}
			\hline & $\gamma_A$ & $\gamma_B$ & $\gamma_C$ & $\lambda_{BC}$\\
			\hline $\lambda$-heavy & 10 & 9 & 0 & 8\\
			$\gamma$-heavy & 10 & 9 & 1 & 7\\
			\hline
		\end{tabular}
		\caption{Example network with dual variables from both analytic cuts. Node A is the source node. Nodes D and E are open facilities, and the arc between B and C is currently closed.}
		\label{FigTwoCut}
	\end{figure}
	
	\subsubsection{Second analytic Benders cut}
	Another algorithm for computing dual variables for an analytic Benders optimality cut is Algorithm \ref{AlgBendOptSecond}. The construction of the second analytic cut is similar to the first, except here we also calculate the distance from each node to their closest open facility, $\bar{D}_i$. Then, before computing the distance of all nodes from the source $k$, we open all arcs. Again, $\gamma_k$ is the length of the shortest path, this time given by $\bar{D}_k$.
	 
	Now, $\gamma_i = \max(0,\min(\gamma_k-D_i,\bar{D}_i))$, that is, it is the smaller of either the saving from opening a facility at node $i$ assuming all arcs are open, or the shortest distance to an open facility under the current configuration of arcs. For nodes further away than $i^*$, $\gamma_i=0$ again. The values of $\lambda_{ij}$ are calculated in the same way as before. These dual variables also form a dual feasible solution for similar reasons as the first. 
	
	\begin{theorem}
		The dual variables computed using Algorithm \ref{AlgBendOptSecond} are dual optimal
	\end{theorem}

	\begin{proof}
	Once again, the constraints are satisfied trivially for $\gamma_i=0$ and by construction for closed arcs. For open arcs, $\lambda_{ij}=0$, so again we need to prove that $\rho_{ij} + \gamma_j - \gamma_i \geq 0$. Also by the properties of shortest distances, we have $\bar{D}_i \leq \bar{D}_j + \rho_{ij}$, or $\rho_{ij}+\bar{D}_j-\bar{D}_i\geq 0$. For these dual variables, when $\gamma_i>0$, we have that $\gamma_i\leq\gamma_k-D_i$ and $\gamma_i\leq \bar{D}_i$. Finally, either $\gamma_j \geq \gamma_k-D_j$ or $\gamma_j \geq \bar{D}_j$.
	
	If $\gamma_j \geq \gamma_k-D_j$, the proof for the first set of dual variables holds since $\gamma_i \leq \gamma_k-D_i$. When $\gamma_j \geq \bar{D}_j$, the fact that $\gamma_i\leq \bar{D}_i$ leads to:
	\begin{align*}
	\rho_{ij} + \gamma_j - \gamma_i &\geq \rho_{ij} + \bar{D}_j - \gamma_i\\
	&\geq \rho_{ij} + \bar{D}_j - \bar{D}_i\\
	&\geq 0
	\end{align*}
	
	So the dual variables satisfy the dual constraints for all arcs. For the same reasons as before, these dual variables give the primal objective value for the current master problem solution, making them a dual optimal solution. 
	\end{proof}
	
	This analytic solution also yields a Pareto-optimal Benders cut, which can be proven in a similar way to the first analytic solution. Where the first was a $\lambda$-heavy cut, the second is called a $\gamma$-heavy cut, since we focus on the saving associated with opening facilities assuming the arcs are already open. We now have (up to) two unique Benders cuts that can be added at each integer solution for each node $k$ and time $t$. If the two cuts are different, then they will hold exactly for different, overlapping sets of master problem solutions.
	
	Consider the network in Figure \ref{FigTwoCut}. Both cuts will be tight for the cases where (B,C) and C are the same, either both open or both closed. However, only the $\lambda$-heavy cut will be tight when C is open but (B,C) is closed, and only the $\gamma$-heavy cut will be tight when (B,C) is open but C is closed. As such, it can be beneficial to add both cuts if possible, as can be seen in the Results section. Note that these Benders cuts are not guaranteed to be different for all solutions, and it is simple to construct a network where both algorithms yield the same Benders cut. As such, one should always take care not to add duplicate Benders cuts to their model, to avoid unnecessarily burdening the solver.
	
	\subsection{Pareto-optimality of the analytically-derived Benders cuts}\label{SubSecPOProof}
	Magnanti and Wong \cite{Magnanti1981} describe the importance of using Pareto-optimal cuts when using Benders decomposition. In this section we show that the analytic Benders cuts are Pareto-optimal. Since a Benders cut is a linear function of the current network configuration, it can be described as $\theta \geq \bar{\theta}(y)$, for $y\in Y$ where $Y$ is the set of all feasible solutions to the master problem. Let the contribution to the objective value for network configuration $y$ be given by $\bar{\theta}^*(y)$. The following definitions are paraphrased from Magnanti and Wong:
	
	\begin{ThmDef}
		A Benders cut $\theta \geq \bar{\theta}^a(y)$ dominates another Benders cut $\theta \geq \bar{\theta}^b(y)$ if $\bar{\theta}^a(y) \geq \bar{\theta}^b(y)$ for all feasible $y\in Y$ and is a strict inequality for at least one feasible $y$.
	\end{ThmDef}

	This definition leads to the following lemma:
	
	\begin{ThmLem}
		If $\theta \geq \bar{\theta}^a(y)$ is dominated by $\theta \geq \bar{\theta}^b(y)$, then for all feasible solutions $y^i$ where $\bar{\theta}^a(y^i)=\bar{\theta}^*(y^i)$, $\bar{\theta}^b(y^i)=\bar{\theta}^*(y^i)$.
	\end{ThmLem} 
	
	This is easy to see, since $\bar{\theta}^b(y^i)\leq\bar{\theta}^*(y^i)$ by definition of being a valid Benders cut, and  $\bar{\theta}^b(y^i)\geq\bar{\theta}^a(y^i)=\bar{\theta}^*(y^i)$ by definition of being a dominating cut.

	\begin{ThmDef}
		A Benders cut $\theta \geq \bar{\theta}^a(y)$ is Pareto-optimal if it is not dominated by any other Benders cuts.
	\end{ThmDef}
	
	One can prove that a Benders cut is Pareto-optimal by assuming that there exists another cut which dominates it, finding enough points $y\in Y$ where the Pareto-optimal cut equals the objective value, and specifying that the dominating cut must also equal the objective value at these points. This leads to all terms of the dominating cut being fixed to those of the Pareto-optimal cut. Thus there are no cuts which dominate the original cut, and it is Pareto-optimal. We now show that our Benders optimality cuts are Pareto-optimal.
	
	\begin{theorem}
		Benders optimality cuts derived using Algorithm \ref{AlgBendOptFirst} are Pareto-optimal
	\end{theorem}
	
	\begin{proof}		
	All Benders cuts for this problem are of the form:
	\begin{equation}
	\theta_{kt} \geq \gamma_k-\sum_{i\in N}\gamma_iW_{it}-\sum_{(i,j)\in A}\lambda_{ij}X_{ijt}
	\end{equation}
	
	Let the dual variables associated with the first analytic Benders cut (i.e. the cut generated by Algorithm \ref{AlgBendOptFirst}) be $\bar{\gamma}_i$ and $\bar{\lambda}_{ij}$. In the current solution, the closest open facility is $i^*$ ($d_{i^*}=\bar{\gamma}_k$).
	
	We assume that the current cheapest facility $i^*$ is not the same as the location, i.e. $i^*\neq k$. If $i^*=k$, then the cut is trivial and not Pareto-optimal. We also assume that there exists another location $j\in N$ such that $d_j<d_{i^*}$, since otherwise the cut will again be trivial and not Pareto-optimal. We begin by defining some partitions of the nodes and arcs of the problem:
	\begin{align*}
	\textbf{F}^o = \{i|W^*_{it} = 1\},&\text{ the set of open facilities, }&\\
	\textbf{F}^c = \{i|W^*_{it} = 0\},&\text{ the set of closed facilities, }&\\
	\textbf{F}^+ = \{i|d_i \geq \bar{\gamma_k}\},&\text{ the set of facilities at equal or greater distance than $i^*$, }&\\
	\textbf{F}^- = \{i|d_i < \bar{\gamma_k}\},&\text{ the set of facilities closer than $i^*$, }&\\
	\textbf{L}^o = \{i|X^*_{ijt} = 1\},&\text{ the set of open links, and}&\\
	\textbf{L}^c = \{i|X^*_{ijt} = 0\},&\text{ the set of closed links }&
	\end{align*}
	
	Now assume there exists a Benders cut using the dual variables $\hat{\gamma}_i$ and $\hat{\lambda}_{ij}$, which dominates the first analytic Benders cut. As they are both Benders cuts, they must both equal the objective value, and thus each other, for the current solution to the master problem. If we open a facility at the source, $k$, the objective value will be zero and the first analytic cut will be tight, so the dominating cut must also equal zero for this solution. This leads to:
	\begin{align*}
	0=&\bar{\gamma}_k - \bar{\gamma}_k - \sum_{\substack{i\in N\\i\neq k}}\bar{\gamma}_iW^*_{it} - \sum_{(i,j)\in A}\bar{\lambda}_{ij}X^*_{ijt}\\
	=& 0 - \sum_{\substack{i\in \textbf{F}^o\\i\neq k}}\bar{\gamma}_i- \sum_{(i,j)\in \textbf{L}^o}\bar{\lambda}_{ij}X^*_{ijt}
	\end{align*}
	Since $\bar{\gamma}_i\geq 0 $ and $\bar{\lambda}_{ij}\geq 0$, we have that $\bar{\gamma}_i=\hat{\gamma}_i=0$ $\forall i\in \textbf{F}^o$ and $\bar{\lambda}_{ij} = \hat{\lambda}_{ij} = 0$ $\forall (i,j) \in \textbf{L}^o$. Note that $i^*\in \textbf{F}^o$, and so $\gamma_{i^*}=0$. Returning to the current solution, the two cuts must equal each other, and for both cuts, either $\gamma_i=0$ or $W_{it}=0$ $\forall i\in N$, and similarly for arcs, so we have that $\hat{\gamma}_k=\bar{\gamma}_k$.
	
	Now, for any other location $i\in N$, if we open a facility at $i$, the first analytic cut will still be tight, so the dominating cut must also be tight. Since the only changes in both cuts is $W_{it}$, we have that $\bar{\gamma}_i=\hat{\gamma}_i$ $\forall i \in N$. All that remains is to show that $\bar{\lambda}_{ij} = \hat{\lambda}_{ij}$ $\forall (i,j) \in \textbf{L}^c$.
	
	For any arc $(i,j)\in \textbf{L}^c$ where $\hat{\gamma}_i\leq \rho_{ijt}$, $\hat{\lambda}_{ij}=0$, which will be tight since even if a facility were opened at $j$, it would still be further away than the closest open facility, so $\bar{\lambda}_{ij}$ will also be zero. The other case where $\hat{\lambda}_{ij}=0$ is when $\hat{\gamma}_j > \hat{\gamma}_i - \rho_{ijt}$, that is, the arc does not create a short-cut in the network. In this case, opening the arc does not change the objective value and the first analytic cut will be tight, so $\bar{\lambda}_{ij}=0$ for all arcs where $\hat{\lambda}_{ij}=0$.
	
	If $\hat{\lambda}_{ij}>0$, then $\hat{\lambda}_{ij}=\hat{\gamma}_i-\hat{\gamma}_j-\rho_{ijt}$ and $\hat{\gamma}_i>\rho_{ijt}$. If we open the arc $(i,j)$ and the facility $j$, then the first analytic cut will be:
	\begin{align*}
	\hat{\gamma}_k-\hat{\gamma}_j-\hat{\lambda}_{ij}=&\hat{\gamma}_k-\hat{\gamma}_j-(\hat{\gamma}_i-\hat{\gamma}_j-\rho_{ijt})\\
	=& \hat{\gamma}_k - \hat{\gamma}_i + \rho_{ijt}
	\end{align*}
	which is the length of the shortest path between $k$ and $i$ plus the length of the arc from $i$ to $j$. Since $\hat{\gamma}_i > \rho_{ijt}$, this will be lower than the original path length, and thus $j$ will be closer than $i^*$ to $k$. So the first analytic cut is tight at these points, and the dominating cut must also be tight. Since $\bar{\gamma}_j=\hat{\gamma}_j$ for all $j\in N$, we have:
	\begin{align*}
	\hat{\gamma}_k - \hat{\gamma}_j-\hat{\lambda}_{ij}=&\bar{\gamma}_k - \bar{\gamma}_j-\bar{\lambda}_{ij}\\
	=&\hat{\gamma}_k - \hat{\gamma}_j-\bar{\lambda}_{ij}\\
	\hat{\lambda}_{ij}=&\bar{\lambda}_{ij}
	\end{align*} 
	
	So $\bar{\lambda}_{ij}=\hat{\lambda}_{ij}$ $\forall (i,j) \in A$, and thus the dominating cut is identical to the first analytic cut. So there are no cuts which dominate the first analytic cut, and thus it is Pareto-optimal.
	\end{proof}

	The same method can be used to prove that the second analytic Benders cut is also Pareto-optimal by selecting different points.

	\subsection{Budget cover inequalities}\label{SubSecBudget}
	In addition to this, we also add inequalities on the budget variables, $U$ and $V$, to potentially tighten the relaxed problem. The budget constraints are effectively a knapsack problem, and as such we can add cover inequalities similar to those described by Gu, Nemhauser and Savelsbergh \cite{Gu1998}. After a solution to the relaxed problem is found, we check, for each time period, which facilities and arcs have been partially or wholly constructed. We sum the variables over all facilities/links and time periods up to and including the current time period, and if this is not an integer value, then some facilities or links have been partially opened. 
	
	$S$ is the sum of facilities/links that have been opened up to this point in time. We then order the facilities/links from cheapest to most expensive to open, and if the sum of opening costs of the first $\lceil S \rceil$ facilities/links is greater than the available budget, we add a new constraint of the form:
	\noindent\begin{minipage}{\textwidth}
		\begin{minipage}[c][5cm][c]{\dimexpr0.5\textwidth-0.5\Colsep\relax}
			\begin{equation}
			\sum_{t'=1}^t\sum_{i \in \bar{S}} U_{it'} \leq \lfloor S \rfloor \label{BC1}
			\end{equation}
		\end{minipage}\hfill
		\begin{minipage}[c][5cm][c]{\dimexpr0.5\textwidth-0.5\Colsep\relax}
			\begin{equation}
			\sum_{t'=1}^t\sum_{\substack{(i,j) \in \bar{S}\\i<j}} V_{ijt'} \leq \lfloor S \rfloor \label{BC2}
			\end{equation}
		\end{minipage}%
	\end{minipage}
	where $\bar{S}$ is the cheapest $\lceil S \rceil$ facilities/links. If it is impossible to open all facilities/links in $\bar{S}$, then turning one off to open another facility/link which is more expensive will not be possible either. Thus, these cuts can be lifted to include all facilities/links more expensive than the most expensive member of $\bar{S}$.

	\section{Warm start}\label{SecWarm}
	One potential problem from applying disaggregated Benders Decomposition to this model is that the initial LP bound is extremely loose. While it is expected that the LP bound will be lower for the Benders master problem because it is a relaxation of the original problem, for this model it is significantly lower. This can be overcome by using a ``warm-start''\cite{deSa2013,McDaniel1977}, which involves solving the linear relaxation of the problem and using the results to add Benders cuts to the master problem. Performing this repeatedly until the bound does not increase substantially, or no more cuts are added, significantly tightens the bound and reduces the runtime of the solver.
	
	This yields significant improvements to the runtime of the program, however it is sometimes more useful to use continuous analogues of the Pareto-optimal analytic Benders cuts in the warm start. Because of this, we analytically construct the dual variables to be used in the pre-cuts. This yields the strongest cuts possible, which can improve the solution speed of the master problem.
	
	\subsection{Feasibility of sub-problems}\label{SubSecWarmFeasibility}
	The first thing to check, as with the main algorithm, is the feasibility of the sub-problems. In the warm start, because all variables are continuous and not integer (or binary), the arcs of the network are allowed to be partially open, and likewise for facilities. As such, it is no longer enough that there be a path to an open facility, instead it must be possible to route all demand to facilities simultaneously. As with the main problem, infeasible sub-problems will only occur when a cycle exists in the network. The IIS feasibility cuts are capable of handling the relaxed problem, and as such are always used in the warm start.
	
	\subsection{Solution of the sub-problems}
	After having ensured the feasibility of the relaxed solutions to the master problem, we solve the flow sub-problems as LP's and extract the paths from these results. When more than one path is required, it is because partially opened arcs or facilities are restricting the flow of demand. In most cases, the longest path will not have any of these restricting factors. 

	If there are $n$ paths in the solution, there will be at least $n-1$ restricting factors. These restricting factors, denoted by the set $\mathcal{C}$, correspond to potential non-zero values for $\gamma$ or $\lambda$ dual variables, and as such there are several constraints on these values. The first is that the sum of the dual variables corresponding to the arcs and final facility of each path must equal the saving from travelling along the path. That is, given a path $p$ of length $L_p$ which ends at node $\textup{dest}_p$, and the set of arcs on that path $A_p$: 
	\begin{equation}
	\gamma_k - L_p = \gamma_{\textup{dest}_p} + \sum_{(a,b)\in A_p}\lambda_{ab} \label{RCP}
	\end{equation}
	
	This equation ensures the reduced cost of each path is zero. In most cases, $\gamma_k$ will be the length of the longest path (opening a facility allows demand from the longest path to be serviced at the source), however in the case where the longest path has restricting factors, the RHS of the above equation will be non-zero for the longest path and $\gamma_k$ will be greater than $L_p$. 
	
	Another condition is that the value of the Benders cut must be equal to the objective value of the sub-problem for the current master problem solution. This is necessary for the dual variables to form a dual-optimal solution. The final condition is that the reduced cost of each arc is non-negative. If a solution is found which satisfies these three conditions, then it is dual optimal, and the dual variables can be used to construct a Benders cut.
	
	If there are $n$ paths and $n-1$ restricting factors, then these dual variables can be calculated directly by solving equation (\ref{RCP}) for all paths simultaneously. In this case, the matrix will be non-singular and thus the values of the dual variables for all restricting factors can be determined. However, as there is much degeneracy in network flow problems, often there will be more than $n-1$ restricting factors for $n$ paths. This occurs when a path has two or more restricting factors which lie only on that path. In this case, one can either determine which $n-1$ factors to use by eliminating any ``extra'' factors, or one can solve the following linear program:
	\begin{equation}
	\text {Minimise } \gamma_k
	\end{equation}
	Subject to:
	\begin{align}
	L_p - \gamma_k + \gamma_{\textup{dest}_p} + \sum_{\substack{(a,b)\in \mathcal{C}\\(a,b)\in A_p}}\lambda_{ab}=0&&\forall p \in P\\
	\gamma_k - \sum_{i\in N}\gamma_i \bar{W}_{it}-\sum_{(i,j)\in A}\lambda_{ij}\bar{X}_{ijt} = \sum_{(i,j)\in A} \rho_{ijt}d_{kt}\bar{Z}_{ijkt}\\	
	\gamma_j-\gamma_i+\lambda_{ij}+\rho_{ijt}d_{kt} \geq 0 && \forall (i,j) \in A
	\end{align}
	
	The effect of constructing analytic warm start Benders cuts this way can be seen in the Results section.

	\section{Results}\label{SecResults}
	We are comparing three different formulations using the public data set from Ghaderi and Jabalameli \cite{Ghaderi2013}. The tests are performed on a PC running Windows 8.1 with an Intel Core i7-3770 quad-core at 3.40GHz and 16GB of RAM. The implementations are written in Python 2.7 and use the Gurobi 6.5 \cite{gurobi} optimisation package. All software used is 64-bit. The maximum runtime for each instance is $50|N||T|$ seconds, where $|N|$ is the number of nodes and $|T|$ is the number of time periods, which is consistent with Ghaderi and Jabalameli \cite{Ghaderi2013}. 
	
	All instances are grouped in threes, where each instance in a group is on the same network, tested over five, 10 and 20 time periods. Each instance has two cases: one where a network already exists, and one where it must be created from scratch. Table \ref{TabInstanceStats} shows the number of nodes, links and time periods of each instance, which can be used to calculate the runtime of each instance.
	
	We start with the straightforward MIP implementation with no improvements. We then compare it to two different implementations of disaggregated Benders decomposition: the first (DBD) is the standard implementation of disaggregated Benders decomposition with regular feasibility cuts, the second (Accelerated DBD) includes the initial time feasibility with IIS feasibility cuts. Both implementations use warm starts without budget cover constraints. Finally, we compare the different possible cuts for the warm start and callback, and evaluate the effectiveness of the budget cover constraints.
	
	Tables \ref{TabEComp1} and \ref{TabNComp1} show the differences between the MIP, DBD and Accelerated DBD implementations on the existing and new network cases respectively. For the majority of cases, Benders decomposition is better than the straightforward MIP, especially on larger networks. There is only a minor benefit to Accelerated DBD over DBD for the existing networks, however on the new network cases Accelerated DBD is a clear winner, with DBD performing worse than the MIP. 
	
	The main reason for the difficulty of the new network cases for the DBD implementation is feasibility of the networks. In the existing network case, the networks are either already feasible or can be made feasible very easily, whereas for the new network case a brand new feasible network must be made from scratch. The standard Benders feasibility cuts used by the DBD implementation are not tight enough to help the algorithm converge, and as such none of the instances are solved to optimality in time, with many failing to find a feasible solution at all. By contrast, the IIS feasibility constraints used in Accelerated DBD allow the problems to be solved significantly faster than the MIP, and solutions were found to previously unsolved instances.

	\begin{table}
		\centering
		\caption{Problem sizes for Ghaderi and Jabalameli instances \cite{Ghaderi2013}}
		\label{TabInstanceStats}
		\begin{tabular*}{0.9\textwidth}{@{\extracolsep{\fill}}|c|ccc||c|ccc||c|ccc|}
			\hline Inst. & N & L & T & Inst. & N & L & T & Inst. & N & L & T\\ \hline 
			TP1 & 20 & 46 & 5 & TP10 & 40 & 162 & 5 & TP19 & 80 & 171 & 5\\  
			TP2 & 20 & 46 & 10 & TP11 & 40 & 162 & 10 & TP20 & 80 & 171 & 10\\ 
			TP3 & 20 & 46 & 20 & TP12 & 40 & 162 & 20 & TP21 & 80 & 171 & 20\\  
			TP4 & 20 & 61 & 5 & TP13 & 60 & 180 & 5 & TP22 & 80 & 280 & 5\\ 
			TP5 & 20 & 61 & 10 & TP14 & 60 & 180 & 10 & TP23 & 80 & 280 & 10\\ 
			TP6 & 20 & 61 & 20 & TP15 & 60 & 180 & 20 & TP24 & 80 & 280 & 20\\ 
			TP7 & 40 & 137 & 5 & TP16 & 60 & 205 & 5 & TP25 & 56 & 200 & 5\\ 
			TP8 & 40 & 137 & 10 & TP17 & 60 & 205 & 10 & TP26 & 56 & 200 & 10\\ 
			TP9 & 40 & 137 & 20 & TP18 & 60 & 205 & 20 & TP27 & 56 & 200 & 20\\ 
			\hline
		\end{tabular*}
	\end{table}
	
\begin{sidewaystable}
	\centering\small
	\caption{Comparison of runtimes and objective values between the CPLEX implementation in Ghaderi and Jabalameli (2013) and our MIP, standard DBD and accelerated DBD on existing network instances. Runs which reached the time limit are reported as Time (gap\%).}
	\label{TabEComp1}
	\begin{tabular*}{\textwidth}{@{\extracolsep{\fill}}|c||r|rrrr|rrr|}
		\hline \multirow{2}{*}{Instance} & \multirow{2}{*}{Objective} & \multicolumn{4}{c|}{Ghaderi et. al. (2013) CPLEX} & \multirow{2}{*}{MIP time (s)} & \multirow{2}{*}{DBD time (s)} & Accelerated\\
		& & Objective & Lower Bound & Gap (\%) & Time (s) & & & DBD time (s) \\ \hline
		TP1E & 14924.87 & 14924.87 & 14924.87 & 0.00 & 4 & 1.24 & 1.87 & 0.87\\  
		TP2E & 29698.23 & 29698.23 & 29698.23 & 0.00 & 14 & 3.00 & 4.26 & 2.12\\ 
		TP3E & 67676.54 & 67676.54 & 67676.54 & 0.00 & 81 & 15.56 & 12.27 & 6.05\\  
		TP4E & 31419.98 & 31419.98 & 31419.98 & 0.00 & 5 & 1.62 & 2.35 & 1.07\\ 
		TP5E & 67884.16 & 67884.16 & 67884.16 & 0.00 & 17 & 5.30 & 5.80 & 1.96\\ 
		TP6E & 158091.31 & 158091.31 & 158091.31 & 0.00 & 77 & 13.41 & 13.00 & 6.00\\ 
		TP7E & 21159.56 & 21159.56 & 21159.56 & 0.00 & 63 & 7.24 & 12.55 & 8.92\\ 
		TP8E & 45581.23 & 45581.23 & 45581.23 & 0.00 & 1758 & 75.48 & 33.45 & 24.22\\ 
		TP9E & 111643.23 & 111643.23 & 110830.16 & 0.73 & 40000 & 253.72 & 95.71 & 62.59\\ 
		TP10E & 19232.89 & 19232.89 & 19232.89 & 0.00 & 289 & 44.15 & 20.82 & 11.92\\ 
		TP11E & 40310.81 & 40310.81 & 39730.11 & 1.44 & 20000 & 664.06 & 250.15 & 43.04\\ 
		TP12E & 99032.15 & 99103.48 & 96797.34 & 2.33 & 40000 & 1456.10 & 501.12 & 95.33\\  
		TP13E & 25110.09 & 25110.09 & 25110.09 & 0.00 & 347 & 31.66 & 28.67 & 19.40\\  
		TP14E & 56892.74 & 56892.74 & 56892.74 & 0.00 & 1479 & 89.81 & 69.40 & 40.39\\  
		TP15E & 151049.50 & 151115.62 & 149518.50 & 1.06 & 60000 & 2778.78 & 356.96 & 324.78\\  
		TP16E & 26973.29 & 26973.29 & 26973.29 & 0.00 & 480 & 36.13 & 40.98 & 24.65\\ 
		TP17E & 58636.16 & 58636.16 & 58636.16 & 0.00 & 16514 & 273.39 & 117.04 & 73.78\\ 
		TP18E & 143986.33 & 144200.68 & 142289.95 & 1.33 & 60000 & 972.10 & 435.64 & 467.36\\  
		TP19E & 33696.53 & 33696.53 & 33696.53 & 0.00 & 5622 & 74.40 & 45.38 & 40.70\\  
		TP20E & 73969.28 & 74677.35 & 73256.40 & 1.90 & 40000 &  481.31 & 156.72 & 132.74\\  
		TP21E & 196674.19 & 198074.52 & 192701.54 & 2.71 & 80000 & 16678.98 & 5884.91 & 4239.84\\  
		TP22E & 30999.44 & 30999.44 & 30999.44 & 0.00 & 1179 & 59.19 & 71.87 & 44.45\\  
		TP23E & 70070.24 & 70092.80 & 69623.72 & 0.67 & 40000 & 622.60 & 235.13 & 156.19\\  
		TP24E & 182700.42 & 183511.97 & 179939.71 & 1.95 & 80000 & 9887.72 & 4536.41 & 2235.90\\ \hline
	\end{tabular*}
\end{sidewaystable}

\begin{sidewaystable}
	\centering\small
	\caption{Comparison of runtimes and optimality gaps between the CPLEX and two heuristic implementations in Ghaderi and Jabalameli (2013) and our MIP, standard DBD and accelerated DBD on new network instances. Runs which reached the time limit are reported as Time (gap\%). If no valid solution was found, they are reported as Time (NS).}
	\label{TabNComp1}
	\begin{tabular*}{\textwidth}{@{\extracolsep{\fill}}|c||r|rrr|rrr|}
		\hline Instance & Best Objective & \multicolumn{3}{c|}{Ghaderi et. al. (2013) Times (s)} & \multirow{2}{*}{MIP Time (s)} & \multirow{2}{*}{DBD Time (s)} & Accelerated\\
		& (Best Bound) & CPLEX & Heuristic & Hybrid SA & & & DBD Time (s)\\ \hline
		TP1N & 20473.40 & Time (7.02\%) & 427 (0.00\%) & 427 (0.00\%) & 283.05 & Time (5.1\%) & 4.13\\ 
		TP2N & 35528.03 & Time (NS) & 371 (0.59\%) & 397 (0.00\%) & 557.54 & Time (1.4\%) & 11.22\\ 
		TP3N & 70973.67 & Time (NS) & 604 (1.10\%) & 645 (0.00\%) & 8619.04 & Time (NS) & 28.22\\ 
		TP4N & 35614.49 & Time (10.66\%) & 1289 (2.23\%) & 1455 (0.00\%) & 2891.17 & Time (9.5\%) & 5.42\\ 
		TP5N & 69150.51 & Time (18.14\%) & 1208 (1.50\%) & 1534 (0.62\%) & Time (9.2\%) & Time (11.2\%) & 11.00\\
		TP6N & 151120.93 & Time (NS) & 1869 (0.63\%) & 4497 (0.28\%) & Time (6.5\%) & Time (13.1\%) & 25.14\\ 
		TP7N & 23120.55 & Time (9.06\%) & 519 (0.00\%) & 519 (0.00\%) & 6500.70 & Time (0.6\%) & 21.06\\ 
		TP8N & 46256.15 & Time (NS) & 2296 (0.00\%) & 2296 (0.00\%) & Time (2.5\%) & Time (NS) & 761.75\\ 
		TP9N & 105113.30 & Time (NS) & 3238 (0.07\%) & 6501 (0.05\%) & Time (2.3\%) & Time (NS) & 916.12\\ 
		TP10N & 19504.02 & Time (NS) & 396 (4.63\%) & 2834 (2.31\%) & 3801.03 & Time (1.6\%) & 59.80\\ 
		TP11N & 37316.52 & Time (8.79\%) & 1822 (4.03\%) & 2605 (0.26\%) & 10577.88 & Time (NS) & 140.99\\ 
		TP12N & 89331.29 & Time (2.68\%) & 2521 (3.16\%) & 9828 (1.17\%) & 10740.89 & Time (NS) & 1417.29\\ 
		TP13N & 28827.77 & Time (9.42\%) & 7428 (0.14\%) & 1758 (0.39\%) & 14074.88 & Time (NS) & 310.84\\ 
		TP14N & 59320.43 & Time (NS) & 14627 (2.34\%) & 26262 (1.28\%) & Time (1.2\%) & Time (NS) & 967.67\\ 
		TP15N & 145023.30 & Time (NS) & Time (3.31\%) & 15819 (1.28\%) & Time (0.4\%) & Time (NS) & 7798.54\\ 
		TP16N & 27741.57 & Time (NS) & 7782 (2.55\%) & 8252 (0.55\%) & Time (2.6\%) & Time (NS) & 572.80\\  
		TP17N & 55740.52 & Time (NS) & 17136 (2.05\%) & 21361 (1.19\%) & Time (1.5\%) & Time (NS) & 2587.47\\  
		TP18N & 132520.16 & Time (6.25\%) & Time (2.29\%) & 13660 (1.12\%) & 39282.31 & Time (NS) & 2753.99\\ 
		TP19N & 37562.93 & Time (19.22\%) & 10805 (0.71\%) & 19255 (0.04\%) & Time (2.8\%) & Time (NS) & 1427.63\\  
		TP20N & 76063.44 & Time (NS) & Time (0.22\%) & 15477 (0.37\%) & Time (1.8\%) & Time (NS) & 13749.00\\ 
		TP21N & $\substack{183881.92\\ (183555.16)}$ & Time (NS) & Time(2.15\%) & 24659 (1.16\%) &  Time (3.5\%) & Time (NS) & Time (0.2\%)\\ 
		TP22N & 33258.34 & Time (6.58\%) & 8332 (2.06\%) & 12874 (0.73\%) & Time (0.2\%) & Time (NS) & 454.37\\  
		TP23N & 70208.93 & Time (7.84\%) & Time (2.09\%) & 31868 (1.25\%) & Time (0.8\%) & Time (NS) & 3177.27\\  
		TP24N & $\substack{171442.41\\ (171036.86)}$ & Time (NS) & Time (2.48\%) & 60596 (1.65\%) & Time (0.8\%) & Time (NS) & Time (0.2\%)\\ \hline
	\end{tabular*}
\end{sidewaystable}


There are three other improvements to the Benders decomposition algorithm we consider here: using analytic Benders cuts (one or two) in the warm start, the budget cover constraints, and using analytic Benders cuts (one or two) during the branch and cut process. We consider the existing and new network cases separately, as the algorithm behaves differently on instances in different cases. After each comparison, we take the best result and keep it for the next comparison.

The main statistic we report is the run time (or achieved optimality gap). We also include the mean and geometric mean of the times for each method across all instances in each case, as well as how many instances each method performed best upon. These help to determine whether or not each method is useful. The first choice we examine is the type of Benders cuts used in the warm start.

\subsection{Warm start}
\begin{sidewaystable}
	\centering\small
	\caption{Comparison of the algorithm's performance when using different cuts in the warm start for existing and new network instances. Non-analytic (NonA), one analytic cut (AnOne) and two analytic cuts (AnTwo). All times are in seconds, instances that are not solved to optimality are marked with Time (optimality gap). Also reported are the Mean time, Geometric Mean (G.M.) time and the number of times each algorithm was best in the instance set.}
	\label{WSCompEN}
	\begin{tabular*}{\textwidth}{@{\extracolsep{\fill}}|c||rrr|rrr|}
		\hline Instance & \multicolumn{3}{c|}{Existing network instances} & \multicolumn{3}{c|}{New network instances}\\
		Number & NonA & AnOne & AnTwo & NonA & AnOne & AnTwo\\ \hline
		TP1 & 0.83 & 0.98 & 0.70& 3.78 & 5.27 & 5.69\\ 
		TP2 & 2.00 & 1.94 & 1.67 & 9.03 & 12.13 & 12.02\\  
		TP3 & 5.84 & 5.26 & 4.49 & 24.76 & 24.82 & 23.83\\  
		TP4 & 0.88 & 0.88 & 0.97 & 3.63 & 8.52 & 9.47\\ 
		TP5 & 1.80 & 2.06 & 2.05 & 7.95 & 12.92 & 16.30\\ 
		TP6 & 5.25 & 4.45 & 4.14 & 16.76 & 28.35 & 29.18\\  
		TP7 & 7.75 & 7.97 & 5.83 & 24.26 & 42.95 & 35.22\\  
		TP8 & 20.82 & 17.03 & 18.75 & 290.13 & 635.40 & 353.16\\
		TP9 & 61.00 & 44.33 & 35.22 & 2136.79 & 1951.36 & 2664.50\\ 
		TP10 & 10.44 & 13.94 & 15.19 & 69.83 & 398.01 & 117.59\\  
		TP11 & 112.67 & 96.11 & 43.70 & 145.50 & 350.51 & 4912.77\\  
		TP12 & 89.47 & 136.03 & 113.27 & 1361.18 & 2130.17 & 1922.54\\
		TP13 & 19.44 & 14.09 & 10.16 & 295.77 & 919.66 & 400.17\\ 
		TP14 & 34.27 & 25.10 & 21.16 & 2085.27 & 735.11 & 1839.12\\ 
		TP15 & 394.73 & 175.58 & 154.17 & 6754.99 & 29861.33 & Time (0.4\%)\\ 
		TP16 & 22.71 & 29.31 & 17.66 & 568.55 & 14900.21 & 699.46\\ 
		TP17 & 66.94 & 63.82 & 43.07 & 1952.92 & 2097.34 & 3922.23\\
		TP18 & 337.36 & 212.04 & 204.33 & 2547.31 & 4818.17 & 6516.81\\ 
		TP19 & 34.16 & 37.84 & 25.66 & 975.16 & 2227.87 & 2512.98\\ 
		TP20 & 243.60 & 210.42 & 108.99 & 9748.52 & Time (1.3\%) & 15566.78\\ 
		TP21 & 4269.29 & 7600.51 & 1435.95 & Time (0.5\%) & Time (8.6\%) & Time (1.5\%)\\ 
		TP22 & 45.06 & 41.09 & 18.35 & 226.32 & 406.74 & 1289.39\\  
		TP23 & 184.25 & 91.47 & 54.17 & 2112.00 & 5488.23 & Time (1.08\%)\\ 
		TP24 & 5875.25 & 827.90 & 440.33 & Time (0.2\%) & Time (0.8\%) & Time (0.4\%)\\ \hline 
		Mean time & 493.57 & 402.51 & 115.83& 7965.48 & 11115.54 & 12602.53\\
		G.M. time & 35.79 & 30.47 & 21.66 & 363.89 & 713.38 & 782.02\\
		Wins & 3 & 1 & 19 & 21 & 2 & 1\\ \hline
	\end{tabular*}
\end{sidewaystable}
The warm start is implemented in the same way for all cases: solve the LP relaxation of the master problem, solve the sub-problems for the resulting values of the relaxed variables and check whether or not a Benders optimality or feasibility cut must be added. If so, add the necessary cuts. This is repeated until the objective value converges or no new cuts are added. All feasibility cuts used are the IIS cuts discussed in Section \ref{SubSecIIS}. There are three options for Benders optimality cuts: non-analytic cuts (NonA) which use the dual variables returned by Gurobi to construct the cut, one analytic cut (AnOne) which is always the $\lambda$-heavy cut, or both analytic cuts (AnTwo), which adds both the $\lambda$-heavy and $\gamma$-heavy cuts if possible. The choice of using the $\lambda$-heavy cut over the $\gamma$-heavy cut for the AnOne implementation was arbitrary, as the problem can be solved using either type of cut and there was no clear-cut winner from our initial experimentation.

Table \ref{WSCompEN} contains the results for the existing network case, where it is clear that using both analytic cuts in the warm start is highly beneficial. The mean and geometric means for AnTwo are significantly lower than the other two methods, and it won in 19 of 24 cases. All cases where it was slower are smaller instances which run in under 100 seconds. For many of the larger instances, it makes a significant improvement.

For the new network case, Table \ref{WSCompEN} shows that using analytic cuts in the warm start becomes a hindrance, with non-analytic cuts winning 21 of 24 times. In one instance, NonA was only 0.93s slower than the best method, and the other two instances where it lost, it is not significantly slower. The mean and geometric mean for NonA are significantly lower than those for AnOne and AnTwo.

For the next comparison, the existing network cases will be run using a warm start with two analytic cuts, and the new network cases with a warm start with non-analytic cuts.

\subsection{Budget cover constraints}

\begin{table}
	\caption{Comparison of algorithm's performance for existing and new network instances, with and without budget cover constraints.}
	\label{BBComp}
	\begin{tabular*}{\textwidth}{@{\extracolsep{\fill}}|c|cc||c|cc|}
		\hline Existing & Without & With & New & Without & With \\ \hline
		TP1E & 0.70 & 0.70 & TP1N & 3.78 & 2.55\\ 
		TP2E & 1.67 & 2.26 & TP2N & 9.03 & 7.72\\ 
		TP3E & 4.49 & 5.38 & TP3N & 24.76 & 20.20\\ 
		TP4E & 0.97 & 0.86 & TP4N & 3.63 & 4.66\\ 
		TP5E & 2.05 & 1.93 & TP5N & 7.95 & 10.79\\ 
		TP6E & 4.14 & 4.63 & TP6N & 16.76 & 17.77\\ 
		TP7E & 5.83 & 5.98 & TP7N & 24.26 & 23.55\\ 
		TP8E & 18.75 & 16.30 & TP8N & 290.13 & 203.77\\ 
		TP9E & 35.22 & 36.05 & TP9N & 2136.79 & 1031.04\\ 
		TP10E & 15.19 & 13.61 & TP10N & 69.83 & 56.38\\ 
		TP11E & 43.70 & 42.65 & TP11N & 145.50 & 196.34\\ 
		TP12E & 113.27 & 97.32 & TP12N & 1361.18 & 1180.41\\ 
		TP13E & 10.16 & 11.16 & TP13N & 295.77 & 321.50\\ 
		TP14E & 21.16 & 25.57 & TP14N & 2085.27 & 2311.46\\ 
		TP15E & 154.17 & 133.59 & TP15N & 6754.99 & 8590.36\\ 
		TP16E & 17.66 & 18.90 & TP16N & 568.55 & 515.83\\ 
		TP17E & 43.07 & 47.96 & TP17N & 1952.92 & 1504.40\\ 
		TP18E & 204.33 & 199.74 & TP18N & 2547.31 & 2312.52\\ 
		TP19E & 25.66 & 26.31 & TP19N & 975.16 & 1578.65\\ 
		TP20E & 108.99 & 73.80 & TP20N & 9748.52 & 8750.23\\ 
		TP21E & 1435.95 & 1610.16 & TP21N & Time (0.5\%) & Time (0.4\%)\\ 
		TP22E & 18.35 & 21.03 & TP22N & 226.32 & 221.79\\ 
		TP23E & 54.17 & 65.08 & TP23N & 2112.00 & 1399.53\\ 
		TP24E & 440.33 & 383.10 & TP24N & Time (0.2\%) & Time (0.2\%)\\ \hline 
		Mean & 115.83 & 118.50 & & 7965.48 & 7919.86 \\
		Geometric Mean & 21.66 & 21.89 & & 363.89 & 343.35 \\
		Wins & 13 & 10 & & 9 & 15 \\ \hline
	\end{tabular*}
\end{table}

The usefulness of the budget cover constraints is more difficult to determine as they make significantly less difference than the choice of warm start optimality cuts, as can be seen in Table \ref{BBComp}. For the existing network case, the means and geometric means of the run times are almost identical, with a roughly 50-50 split on the number of wins. For the new network case, they yield a slight improvement, with lower mean and geometric means of the run times, and winning over 60\% of the time. 

The difference between the existing network and new network cases is that in the new network case the network must be built from scratch and the budget is twice as big, so there are many more arcs and facilities which will be fractionally open in the warm start solutions. For the existing network case, many arcs and facilities are already open, and so there will be fewer fractional variables in the warm start solution, meaning fewer budget cover constraints will be added.

For the next test, the existing network instances will be run with two analytic cuts and no budget cover constraints in the warm start, and the new network instances with non-analytic cuts and with budget cover constraints.

\subsection{Benders optimality cuts}

\begin{sidewaystable}
	\centering\small
	\caption{Comparison of different Benders optimality cuts for the existing network instances. All implementations use the same code (two-cut analytic warm start and no budget cover constraints) with the exception of the callback used to implement Benders optimality cuts. The different callbacks use non-analytic duals (CNonA), one analytic cut (CAnOne) and two analytic cuts (CAnTwo). All times are in seconds, instances that are not solved to optimality are marked with Time (optimality gap). Also reported are the Mean time, Geometric Mean (G.M.) time and the number of times each algorithm was best in the instance set.}
	\label{CBCompEN}
	\begin{tabular*}{\textwidth}{@{\extracolsep{\fill}}|c||rrr|rrr|}
		\hline Instance & \multicolumn{3}{c|}{Existing network instances} & \multicolumn{3}{c|}{New network instances}\\
		Number & CNonA & CAnOne & CAnTwo & CNonA & CAnOne & CAnTwo\\ \hline
		TP1 & 0.70 & 0.65 & 0.76 & 2.55 & 2.98 & 2.89\\ 
		TP2 & 1.67 & 1.84 & 1.39 & 7.72 & 7.09 & 7.86\\ 
		TP3 & 4.49 & 4.32 & 4.46 & 20.20 & 22.43 & 20.60\\  
		TP4 & 0.97 & 1.28 & 0.85 & 4.66 & 3.69 & 3.22\\  
		TP5 & 2.05 & 2.03 & 2.16 & 10.79 & 8.99 & 10.54\\ 
		TP6 & 4.14 & 3.86 & 4.27 & 17.77 & 16.11 & 16.44\\  
		TP7 & 5.83 & 5.67 & 5.47 & 23.55 & 26.18 & 22.68\\ 
		TP8 & 18.75 & 14.43 & 15.00 & 203.77 & 236.09 & 274.26\\  
		TP9 & 35.22 & 26.65 & 26.41 & 1031.04 & 710.72 & 1683.92\\  
		TP10 & 15.19 & 14.26 & 14.19 & 56.38 & 50.86 & 187.88\\
		TP11 & 43.70 & 37.17 & 37.35 & 196.34 & 198.83 & 197.57\\
		TP12 & 113.27 & 155.75 & 122.36 & 1180.41 & 810.04 & 958.81\\  
		TP13 & 10.16 & 6.78 & 6.74 & 321.50 & 511.30 & 467.62\\  
		TP14 & 21.16 & 15.55 & 15.27 & 2311.46 & 1362.79 & 987.09\\  
		TP15 & 154.17 & 128.18 & 114.79 & 8590.36 & 16371.58 & 16220.59\\ 
		TP16 & 17.66 & 15.22 & 13.60 & 515.83 & 439.66 & 485.31\\  
		TP17 & 43.07 & 36.54 & 34.15 & 1504.40 & 1254.36 & 1315.31\\ 
		TP18 & 204.33 & 140.17 & 163.00 & 2312.52 & 3934.44 & 2962.72\\ 
		TP19 & 25.66 & 17.20 & 17.78 & 1578.65 & 4851.20 & 2123.34\\ 
		TP20 & 108.99 & 45.61 & 49.23 & 8750.23 & 25858.42 & 7123.92\\ 
		TP21 & 1435.95 & 1243.31 & 1568.76 & Time (0.4\%) & Time (0.2\%) & Time (0.3\%)\\ 
		TP22 & 18.35 & 11.73 & 12.04 & 221.79 & 455.68 & 286.42\\  
		TP23 & 54.17 & 36.32 & 31.69 & 1399.53 & 973.76 & 1840.23\\  
		TP24 & 440.33 & 373.03 & 515.29 & Time(0.2\%) & Time (0.2\%) & Time (0.2\%)\\ \hline
		Mean time & 115.83 & 97.40 & 115.71 & 7919.86 & 9080.08 & 8208.89 \\
		G.M. time & 21.66 & 18.10 & 17.95 & 343.35 & 379.61 & 377.50 \\
		Wins & 1 & 12 & 11 & 9 & 10 & 4\\ \hline
	\end{tabular*}
\end{sidewaystable}

The last choice to make is the type of Benders optimality cuts to use in the main callback. These are the cuts that will be added during the branch and bound process to solve the problem to optimality. Again, the options are the non-analytic dual variables returned by Gurobi (CNonA), or the analytically constructed dual variables. For the one cut case (CAnOne), the $\lambda$-heavy cut will be used for the same reason as the warm start.

For the existing network case (seen in Table \ref{CBCompEN}), analytic cuts are clearly beneficial, with CNonA only winning on one of the 24 instances (TP12E). The difference between analytic and non-analytic run times was also clear, with non-analytic having a higher geometric mean than the analytic methods. As for whether to use one or two analytic cuts, it seems to make little difference with very similar run times and an even split of the wins. CAnOne won all but one of instances 18-24 (it lost to CAnTwo on TP23E by 5 seconds), however it was often very close.

Table \ref{CBCompEN} shows that for the new network case, using analytic cuts was again less beneficial, with CAnOne having a lower mean and geometric mean of the run times. CAnTwo only won a few times, with the rest of the instances being split between CNonA and CAnOne. For instances 18 and above, CNonA beat CAnOne in all but two cases (TP21N and TP23N) and tied in one (TP24N).

\section{Discussion}\label{SecDiscussion}
	
\subsection{Disaggregation level}\label{SubSecDisaggLevel}
In this problem it is possible to disaggregate over two different sets: the source of each demand and each separate time period. We show here that it is best to disaggregate by both sets at the same time. Disaggregation of sub-problems, and thus Benders cuts, always results in tighter bounds. These tighter bounds allow the master problem to be solved more quickly. The trade-off is that having more sub-problems can take longer to solve, particularly if there are overheads associated with those sub-problems. In this problem, the most sub-problems we solve are 1600, which is acceptable considering the speed increases we obtain from this. In other problems, the number of sub-problems may enter the hundreds of thousands, at which point even the smallest overheads will start to add up.
	
A specific example is data set TP9E, which we can compare results for if we disaggregate only by nodes, only by time and by both nodes and time. For this instance there are 40 nodes and 20 time periods. Table \ref{TabDisaggregation} shows the number of sub-problems (S.P.'s) solved, the total time spent solving sub-problems and the total time spent solving the entire problem for the different levels of disaggregation.
	
	\begin{table}
		\caption{Comparison of solution times for problem TP9E with different levels of disaggregation}\label{TabDisaggregation}
		\begin{tabular}{|c|ccc|}
			\hline Disaggregation level & \# S.P's solved & S.P. cumulative time (s) & Master solve time (s) \\ \hline
			Time only & 580 & 24.74 & 283.35\\
			Node only & 760 & 14.72 & 232.77\\
			Node and Time & 11200 & 9.48 & 101.47\\ \hline
		\end{tabular}
	\end{table}
	
We can see that disaggregating more leads to smaller sub-problems which solve significantly faster. The average solve time for each sub-problem is 43ms, 19ms and 0.85ms for time only, nodes only and both, respectively. Even though many more sub-problems must be solved when disaggregating by both nodes and time, the cumulative time spent solving them is less, and the tighter cuts provided by disaggregation leads to a faster solve time of the master problem.
	
\section{Conclusion}\label{SecConclusion}
Disaggregated Benders decomposition with lazy constraints is an effective method for solving the DUFLNDP if implemented properly. Adding constraints that enforce feasibility to avoid relying upon Benders feasibility cuts, and using a warm start are good ways of improving the effectiveness of the solver. Analytically derived Pareto-optimal Benders cuts can also be beneficial in some cases. For this particular problem, it is the disaggregation of the sub-problems and IIS feasibility cuts which provide the impressive speed increase, which has allowed us to solve almost all instances to optimality within the time limits. In the future, we would like to generalise this approach to a wide range of network design and facility location problems where similar techniques are beneficial.

	\section{References}
	\bibliography{DUFLNDP}

\end{document}